\documentclass[10pt,a4paper]{article}
\usepackage{setspace,framed,pifont,dingbat} 
\usepackage{amsmath,soul,color,amssymb,graphicx,rotating,colortbl,xcolor} 
\usepackage[amsmath,thmmarks,amsthm,framed]{ntheorem} 
\usepackage{typearea} 
\usepackage[T1]{fontenc} 
\usepackage{lmodern}
\usepackage[latin1]{inputenc} 
\usepackage{slashed}
\usepackage{mathtools}
\usepackage{mathrsfs}
\usepackage{titletoc,titlesec} 
\usepackage{microtype}
\usepackage{floatflt}
\usepackage[colorlinks=true,linkcolor=black, urlcolor=myblue2, citecolor = black]{hyperref} 
\usepackage{listings}
\usepackage{tikz}
\usepackage{subfigure}
\usepackage{graphicx}
\usepackage{url}
\usepackage[a4paper,left=2cm,right=2cm,top=2cm,bottom=4cm,bindingoffset=5mm]{geometry}

\titleformat{\chapter}[display] {\normalfont\Large\filcenter\normalfont}
{\titlerule[1pt]\vspace{1pt}%
\titlerule\vspace{1pc}%
\LARGE\MakeUppercase{\chaptertitlename} \thechapter} {1pc}{
\titlerule\vspace*{1pc}\Huge}
\titlespacing{\chapter}{0pt}{-3em}{3em}
\theoremstyle{break}

\theorembodyfont{\normalfont}

\newtheorem{thm}{Theorem}
\newtheorem{lem}{Lemma}

\newtheorem{defn}{Definition}
\newtheorem{rem}{Remark}

\newtheorem{cor}{Corollary}
\newtheorem{prop}{Proposition}

\newcommand{\dx}{\,\text{d}}

\newcommand{\qb}{\mathbf{b}}
\newcommand{\qd}{\mathbf{d}}
\newcommand{\qe}{\mathbf{e}}
\newcommand{\qk}{\mathbf{k}}

\newcommand{\qm}{\mathbf{m}}
\newcommand{\qn}{\mathbf{n}}
\newcommand{\qp}{\mathbf{p}}

\newcommand{\qu}{\mathbf{u}}
\newcommand{\qx}{\mathbf{x}}
\newcommand{\qy}{\mathbf{y}}

\newcommand{\FT}{\mathcal{F}} 
\newcommand{\ZT}{\mathcal{Z}} 
\newcommand{\QT}{\mathcal{Q}} 

\newcommand{\QH}{\mathbb{H}}
\newcommand{\Sc}{\text{Sc}}
\newcommand{\VEC}{\text{Vec}}
\begin{document}
\
\begin{center}
\LARGE{\textbf{
Relaxed quaternionic Gabor expansions at critical density
}}
\end{center}
\begin{center}
Stefan Hartmann
\end{center}

\vspace{2em}
\begin{abstract}
Shifted and modulated Gaussian functions play a vital role in the representation of signals. We extend the theory into a quaternionic setting, using two exponential kernels with two complex numbers. As a final result, we show that every continuous and quaternion-valued signal $f$ in the Wiener space can be expanded into a unique $\ell^2$ series on a lattice at critical density $1$, provided one more point is added in the middle of a cell. We call that a \emph{relaxed Gabor expansion}.
\end{abstract}

\section{Introduction}
Phase, as have shown A. V. Oppenheim and J. S. Lim in the 1980s, codes all relevant details of an image \cite{OppenheimLim}. However, reasonable phase terms for higher dimensional settings, e.g. for 3-D image stacks or for 6 + 1 D (3D color films, e.g. from confocal microscopy) have long been an open question. Among the different concepts in 1999 T. Bülow introduced the hypercomplex signal in \cite{BULOWPHD}. The concept of hypercomplex signal was afterwards widely used in image processing and even applied to movies (see for instance the paper of W. Chan, H. Choi and R. Baraniuk \cite{Baraniuk, Baraniuk2}). Bülow's approach is based on a specific type of a quaternionic Fourier transform. There exist several possibilities how one can define a quaternionic Fourier transform, see e.g. T. A. Ells article \cite{TAELL}. Among these, the right sided version has been already widely discussed by E. Hitzer, M. Bahri et al. \cite{Bahri-Quat_Gabor, HITZER-Uncertainty, Hitzer-Uncertainty2} and others. Bülow's concept is based on the two sided version, where the signal is wedged between two exponential kernels. This idea leads to more symmetric properties \cite{BalianLowforWQFT}. For the mathematical background we refer to \cite{Heise, SpringerPaula}. But for practical application using only a Fourier expansion there is a major drawback. The quaternionic Fourier transform lacks an important property, this is the locality of phase behavior. In order to overcome this issue we multiply the signal by a real-valued window function with compact support and then apply the quaternionic Fourier transform. That way we obtain a representation both in space and phase and the so-called windowed quaternionic Fourier transform (WQFT). But in applications we can only work with a discrete set in the parameter space such as a lattice. A discretization of the WQFT leads into Gabor frame theory. An important question in that theory is, how dense we have to sense the signal, or in other words, which is the minimum size a cell in the lattice should have. In the non-quaternionic theory it is well-known, that under certain requirements to the window function, the area of a cell has to be smaller or equal than $1$ for a perfect reconstruction, whereas the case of greater than $1$ never leads into a perfect reconstruction. For the special case of the Gaussian as window function, the case equal to $1$, also referred to as \emph{critical density} is of main interest, since the set of coefficients is complete, but does not constitute a Gabor frame. V. Palamodov showed that there exist a unique $\ell^2$ expansion into shifted and modulated Gaussian windows for the critical density, provided one point is added in the middle of a cell \cite{Palamodov}. Unfortunately, the same can not be said about the quaternionic case. Up to now there is no result regarding the critical density. The principal goal of this work is to construct Gabor expansions with critical density, so-called \emph{relaxed Gabor expansions}, which is essential for the practical application of quaternionic Gabor expansions in Image Processing.\\
The paper is organized as follows: In Section 2 we briefly review some results on quaternionic algebra, the used function spaces and the WQFT. In Section 3 we introduce Gabor methods, such as the Gabor transform and the Gabor series. The fourth section treats the quaternionic Zak transform, a natural extension of the Zak transform, which is the main tool for the results. In the last section we show that a quaternionic signal has a unique $\ell^2$ expansion at critical density, when one more point is added to a lattice. Furthermore, a formula for the coefficients of the Gabor expansion is obtained.

\section{Preliminaries}
\subsection{Quaternions}
Quaternions are a non-commutative extension of the complex numbers into four dimensions.
The \emph{quaternion algebra} $\QH$ is given by
\[ \QH = \{ q \,|\, q = q_0 + i q_1 + j q_2 + kq_3, \text{\  with\  } q_0, q_1, q_2,q_3 \in \mathbb{R}\}, \]
where the elements $i,j,k$ satisfy
\[ ij = -ji = k, \qquad i^2 = j^2 = k^2 = ijk = -1. \]

We can write a quaternion $q$ also as a sum of a \emph{scalar} $\Sc[q] = q_0\in \mathbb{R}$ and a three-dimensional vector $\VEC[q] = i q_1 + j q_2 + k q_3$, which is often called \emph{pure quaternion}. The conjugate is an automorphism on $\mathbb{H}$, given by $q \mapsto \overline{q} = \Sc[q] - \VEC[q]$. The \emph{modulus} $|q|$ of a quaternion is defined by
\begin{equation}\label{eq:mod}  |q| = \sqrt{q \cdot \overline{q}} = \sqrt{q_0^2 + q_1^2 + q_2^2 + q_3^2}.  \end{equation}
It also holds $\Sc |q| \leq |q|$ and $q \overline{q} = |q|^2$. The inverse of a non-zero quaternion $q$ is given by $\frac{\overline{q}}{|q|^2}$. The non-commutativity of the quaternions means that in general $p^{-1}q \neq q p^{-1}$, such that $\frac{q}{p}$ in general does not make sense. Nevertheless, for an easier reading we will use the latter notation from time to time but we will make clear what is meant by it in each case. The following properties also should be brought up
\[ \overline{p \cdot q} = \overline{q} \cdot \overline{p} \qquad \text{and} \qquad |p \cdot q|^2 = |p|^2 \cdot |q|^2.  \]
Although quaternions are non commutative, the \emph{cyclic multiplication} has to be mentioned, since it is going to be a useful tool.
\begin{lem}[Cyclic Multiplication \cite{Sprssig}]
For all $q,r,s \in \QH$ holds
\begin{equation}\label{eq:SC} \Sc[q  r  s] = \Sc[r s q] = \Sc[s q r]. \end{equation}
\end{lem}

\subsection{Function spaces}
In the following we use the space $L^2(\mathbb{R}^2,\mathbb{H})$, an immediate generalization of the Hilbert space of all square-integrable functions, which consists of all quaternion-valued functions $f: \mathbb{R}^2 \to \mathbb{H}$ with finite norm
\[ \|f\|_2 = \left(\int_{\mathbb{R}^2} |f(\qx)|^2 \dx^2 \qx \right)^{1/2} < \infty \]
where $\dx^2 \qx = \dx x_1 \dx x_2$ represents the usual Lebesgue measure in $\mathbb{R}^2$. The $L^2$-norm is induced by the symmetric real scalar product
\begin{equation}\label{eq:Sc}
\langle f, g \rangle = \frac{1}{2}\int_{\mathbb{R}^2} (g(\qx) \overline{f(\qx)} + f(\qx) \overline{g(\qx)})\dx^2 \qx = \Sc \int_{\mathbb{R}^2} f(\qx) \overline{g(\qx)} \dx^2 \qx \end{equation}
which makes the space a real linear space. 
\begin{rem}[\cite{Hitzer-QFT}]
It is also possible to define a quaternion-valued inner product on $L^2(\mathbb{R}^2,\mathbb{H})$ with
\[ (f,g) = \int_{\mathbb{R}^2} f(\qx) \overline{g(\qx)} \dx^2 \qx \]
in which case one obtains a \emph{left Hilbert module}.
One has to be careful with the properties, since for $p \in \mathbb{H}$ holds
\[ (p f, g) = p (f,g) \qquad \text{but} \qquad (f, g p) = (f, g) \overline{p}. \]
Both inner products lead to the same norm.  
\end{rem}

We denote the two-dimensional cube $[0,1]^2$ briefly by $Q$. Moreover we need the notion of a Wiener-Amalgam space in the quaternionic setting. The Wiener-Amalgam space comes from the periodization of functions, the main idea for the definition arises from a joint categorization of local and global behavior. 
\begin{defn}[Wiener-Amalgam space \cite{Grchenig2001}]\label{defn:wiener}
A function $f \in L^\infty(\mathbb{R}^2, \mathbb{H})$ belongs to the Wiener-Amalgam Space $W = W(\mathbb{R}^2, L^\infty)$ if
\[ \|f\|_{W} = \sum_{n \in \mathbb{Z}^d} \|f \cdot T_n \mathcal{X}_{Q}\|_{\infty} \]
is finite.
\end{defn}
The subspace of continuous functions is denoted by $W_0$.

\subsection{The windowed quaternionic Fourier transfom}
For the analysis of signals, especially for results on the phase of an image, the windowed quaternionic Fourier transform will be our main tool. It is useful to utilize two exponential factors, which can be put either on the same side of the signal (one sided) or they can sandwich the signal (two sided). Since we are interested in applications using the hypercomplex signal, we use the two sided version which has some more symmetric features than the one sided one. The reader should note that these two Fourier transforms differ, owing to the non-commutativity, when a quaternion-valued signal is analyzed.  
\begin{defn}[Two sided quaternionic Fourier transform (QFT) \cite{BalianLowforWQFT}]
The \emph{two sided quaternionic Fourier transform} of $f \in L^2(\mathbb{R}^2,\mathbb{H})$ is the function $\FT_q(f)\,:\, \mathbb{R}^2 \to \mathbb{H}$ defined by
\[ \FT_q(f)(\omega) = \widehat{f}(\omega) = \int_{\mathbb{R}^2} \exp(-2 \pi i x_1 \omega_1) f(\qx) \exp(-2 \pi j x_2 \omega_2) \dx^2 \qx. \]
\end{defn}

\begin{rem}[\cite{BalianLowforWQFT}]
We can reconstruct the signal $f$ from its QFT with
\begin{equation}\label{eq:REC} f(\qx) = \FT_q^{-1}(\FT_q f)(\qx) = \int_{\mathbb{R}^2} \exp(2 \pi i x_1 \omega_1) \widehat{f}(\omega) \exp(2 \pi j x_2 \omega_2) \dx^2 \omega. \end{equation}
\end{rem}

There exists also a Parseval Theorem for the QFT, which is given in the following Theorem.  
\begin{thm}[Parseval Theorem \cite{BalianLowforWQFT}]
For $f, g \in L^2(\mathbb{R}^2,\mathbb{H})$ holds
\[ \langle f, g \rangle = \langle \widehat{f}, \widehat{g} \rangle. \]
\end{thm}
\begin{proof}
We use the reconstruction formula \eqref{eq:REC} and the cyclic multiplication property \eqref{eq:SC}
\begin{align*}
\langle f, g \rangle 
&= \Sc \int_{\mathbb{R}^2} f(\qx) \overline{g(\qx)} \dx^2 \qx\\
&= \Sc \int_{\mathbb{R}^2} \int_{\mathbb{R}^2} \exp(2 \pi i x_1 \omega_1) \widehat{f}(\omega) \exp(2 \pi j x_2 \omega_2) \dx^2 \omega \,\overline{g(\qx)} \dx^2 \qx\\
&= \Sc \int_{\mathbb{R}^2} \widehat{f}(\omega) \int_{\mathbb{R}^2} \exp(2 \pi j x_2 \omega_2) \overline{g(\qx)} \exp(2 \pi i x_1 \omega_1) \dx^2 \omega \dx^2 \qx\\
&= \Sc \int_{\mathbb{R}^2} \widehat{f}(\omega) \int_{\mathbb{R}^2} \overline{\exp(- 2 \pi i x_1 \omega_1) g(\qx) \exp(- 2 \pi j x_2 \omega_2)} \dx^2 \qx \dx^2 \omega\\
&= \Sc \int_{\mathbb{R}^2} \widehat{f}(\omega) \overline{\widehat{g}(\omega)} \dx^2 \omega \\
&= \langle \widehat{f}, \widehat{g} \rangle. 
\end{align*}
\end{proof}
The cyclic multiplication is essential for the proof. So far there is no Parseval Theorem for the quaternion-valued inner product. As a simple corollary we obtain the \emph{Plancherel formula}, which states that the energy of a signal is preserved under the quaternionic Fourier transform
\begin{equation}\label{eq:ParsevalQ} \|f\|^2_2 = \|\widehat{f}\|^2_2.\end{equation}
Nevertheless it is in general not possible to obtain local information of the frequency behavior of the signal $f$ only by looking at $\widehat{f}$. In order to obtain local information of the frequency behavior, there are several possibilities. One could be to cut the signal into pieces and then apply the QFT. Another idea is to multiply the signal with a real-valued function of compact support and then apply the QFT. The first method leads into the Gibbs-phenomena, therefore we use the latter.  
\begin{defn}[Windowed quaternionic Fourier transform \cite{BalianLowforWQFT}]
The windowed quaternionic Fourier transform (WQFT) of $f \in L^2(\mathbb{R}^2, \mathbb{H})$ with respect to a non-zero real-valued window function $g \in L^2(\mathbb{R}^2, \mathbb{R})$ is defined as
\begin{equation}\label{eq:WQFT} \QT_g f(\qb, \omega) = \int_{\mathbb{R}^2} \exp(-2 \pi i x_1 \omega_1) f(\qx) g(\qx-\qb) \exp(-2 \pi j x_2 \omega_2) \dx^2 \qx.
\end{equation}
\end{defn}
We can reconstruct the signal from its WQFT.
\begin{prop}[Reconstruction formula \cite{BalianLowforWQFT}]
Let $g \in L^2(\mathbb{R}^2, \mathbb{R})$ be a non-zero real valued window function. Then every $f \in L^2(\mathbb{R}^2,\mathbb{H})$ can be fully reconstructed by
\begin{align}\label{eq:RecWQFT}
f(\qx) = \frac{1}{\|g\|^2_2}\int_{\mathbb{R}^2}\int_{\mathbb{R}^2} \exp(2 \pi i x_1 \omega_1) \QT_g f(\qb, \omega)g(\qx - \qb) \exp(2 \pi j x_2 \omega_2) \dx ^2 \omega \dx^2 \qb.
\end{align}
\end{prop}

By looking at the definition of the WQFT \eqref{eq:WQFT}, we can see that the signal $f$ stays untouched. The translation and modulation work on the window $g$, which can be interpreted as a $L^2$-kernel. This gives the first motivation to rewrite the WQFT with the help of an inner product. The problem is that we have to find a way to express the kernel. In order to do that, we use the concept of carriers, introduced by Shapiro et al. \cite{Shapiro}.

\begin{defn}[Carrier \cite{Zakquat, Shapiro}]
We define for two quaternions $p,q \in \mathbb{H}$ the \emph{right} $C_r$ and \emph{left} $C_l$ \emph{carrier operator}
\[ C_r(p)q = q p \qquad \text{and} \qquad q C_l(p) = p q. \] 
\end{defn}

\begin{lem}[Properties of the carrier]
We have for the carriers the following properties with $p \in \mathbb{H}$ \\
(a)\ $\overline{C_r(p)} = C_l(\overline{p})$ \  and \  $\overline{C_l(p)} = C_r(\overline{p})$,\\
(b)\ $C_r(p)1 =  1C_l(p) = p$.
\end{lem}

Now that we are equipped with the necessary tools we can generalize the translation and modulation for our setting.
\begin{defn}
For $\qb, \omega \in \mathbb{R}^2$ we define the following operators
\[ T_{\qb} g(\qx) = g(\qx - \qb) \] and
\[ M_\omega g(\qx)  = \exp(2 \pi j \omega_2 x_2) g(\qx) C_r(\exp(2 \pi i \omega_1 x_1)). \]
\end{defn}
This allows us to write the WQFT in the following manner
\[ \QT_g f(\qb, \omega) = (f, M_\omega T_{\qb} g). \]
For the rest of this paper we employ the two-dimensional Gaussian function as window function, since it minimizes the uncertainty principle, i.e. it has the best localization in space and frequency. For shorter notation we use 
\[ \qe_{\lambda}(\qx) = M_\omega T_{\qb} \exp(- \pi |\qx|^2) = \exp(2 \pi j x_2 \omega_2) \exp(- \pi |\qx - \qb|^2) C_r(\exp(2 \pi i x_1 \omega_1)) \]
and
\[ \qe'_{\lambda}(\qx) = \overline{M_{-\omega}} T_{\qb} \exp(- \pi |\qx|^2) = C_l(\exp(2 \pi i x_1 \omega_1)) \exp(- \pi |\qx - \qb|^2) \exp(2 \pi j x_2 \omega_2) \]
with $\lambda = (\qb, \omega)$ being a four-dimensional point in $\mathbb{R}^2 \times \mathbb{R}^2$.

\section{Gabor methods}
In this section we will discuss necessary properties of the Gabor transform introduced in the previous section. Let us start with the continuous Gabor transform. 
\subsection{Gabor transform}
With the reconstruction formula \eqref{eq:RecWQFT} for the WQFT, we get for the special case $g(\qx) = \exp(- \pi |\qx|^2)$ the following expression
\[ f(\qx) = \int_{\mathbb{R}^2} \int_{\mathbb{R}^2} \exp(2 \pi i x_1 \omega_1) (f, M_\omega T_\qb g) g(\qx - \qb) \exp(2 \pi j x_2 \omega_2) \dx^2 \omega \dx^2 \qb = \int_{\mathbb{R}^2} \int_{\mathbb{R}^2} (f, \qe_{\lambda}) \qe'_{\lambda} \dx \lambda. \]
Since every quaternionic function can be split into a sum of real functions, we get the explicit formula:
\begin{align*}
f(\qx) &= \int_{\mathbb{R}^2} \int_{\mathbb{R}^2} \exp(2 \pi i x_1 \omega_1) (f, \qe_\lambda) \exp(-  \pi |\qx - \qb|^2) \exp(2 \pi j x_2 \omega_2) \dx^2 \omega \dx^2 \qb \\ 
&= \int_{\mathbb{R}^2} \int_{\mathbb{R}^2} \exp(2 \pi i x_1 \omega_1) \left[ \langle f, \qe_\lambda \rangle + i \langle - i f, \qe_\lambda \rangle + \langle - j f, \qe_\lambda \rangle  j + i \langle - k f , \qe_\lambda \rangle j \right]\\
& \qquad \qquad \exp(- \pi | \qx - \qb|^2) \exp(2 \pi j x_2 \omega_2) \dx^2 \omega \dx^2 \qb. 
\end{align*}
This formula helps us in the next lemma. For $f \in L^2(\mathbb{R}^2, \mathbb{H})$ the function $\lambda \mapsto (f, \qe_\lambda)$ defined in $\mathbb{R}^2 \times \mathbb{R}^2$ is called \emph{Gabor transform}. With the next Lemma we show that the Gabor transform is unitary.

\begin{lem}
For an arbitrary function $f \in L^2(\mathbb{R}^2, \mathbb{H})$ holds
\[ \int_{\mathbb{R}^2} \int_{\mathbb{R}^2}|\langle f, \qe_\lambda \rangle|^2 \dx^2 \qb \dx^2 \omega = \frac{1}{4} \|f\|^2_2 \]
and  
\[ \int_{\mathbb{R}^2} \int_{\mathbb{R}^2} \left| \int_{\mathbb{R}^2} \exp(-2 \pi i x_1 \omega_1) f(\qx) \exp(- \pi | \qx - \qb |^2) \exp(-2 \pi j x_2 \omega_2) \dx^2 \qx \right|^2 \dx^2 \qb \dx^2 \omega = \int_{\mathbb{R}^2} \int_{\mathbb{R}^2} |(f, \qe_\lambda)|^2 \dx^2 \lambda =  \|f\|_2^2. \]
\end{lem}
\begin{proof}
Since 
\begin{align}
&\int_{\mathbb{R}^2} |\langle f, \qe_\lambda \rangle|^2 \dx^2 \omega \nonumber \\
&= \int_{\mathbb{R}^2} \Sc \left|  \int_{\mathbb{R}^2} \exp(2 \pi i x_1 \omega_1) f(\qx) \exp(- \pi |\qx - \qb|^2) \exp(2 \pi j x_2 \omega_2) \dx^2 \qx \right|^2 \dx^2 \omega \nonumber \\
&= \int_{\mathbb{R}^2} \Sc \left| \FT_q (f_\qb)(-\omega) \right|^2 \dx^2 \omega \label{eq:Plancherel2} \\
&= \int_{\mathbb{R}^2} \Sc \left| f_\qb(\omega) \right|^2 \dx^2 \qx \label{eq:fqb}
\end{align}
with $f_\qb(\qx) = f(\qx) \exp(- \pi |\qx - \qb|^2)$. We used Plancherel Theorem \eqref{eq:ParsevalQ} in \eqref{eq:Plancherel2}. Integrating \eqref{eq:fqb} over $\qb$ and changing variables $\qy = \qx - \qb$ leads to
\begin{align*}
& \int_{\mathbb{R}^2} \Sc \int_{\mathbb{R}^2} | f(\qx) \exp(- \pi |\qx - \qb|^2)|^2 \dx^2 \qx \dx^2 \qb \\
&= \int_{\mathbb{R}^2} \Sc \int_{\mathbb{R}^2} |f(\qx) \exp(- \pi |\qy|^2)|^2 \dx^2 \qx \dx^2 \qy\\
&= \int_{\mathbb{R}^2} |\exp(- \pi |\qy|^2)|^2 \dx^2 \qy \quad  \Sc \int_{\mathbb{R}^2} |f(\qx)|^2 \dx^2 \qx\\
&= \frac{1}{4} \|f\|_2^2.  
\end{align*} 
Finally we obtain
\[ (f, \qe_\lambda) = \langle f, \qe_\lambda \rangle + i \langle - i f, \qe_\lambda \rangle +  \langle - j f, \qe_\lambda \rangle j+ i \langle - k f, \qe_\lambda \rangle j \] 
to find
\[ \int_{\mathbb{R}^2} \int_{\mathbb{R}^2} (f, \qe_\lambda) \dx^2 \qb \dx^2 \omega = \|f\|^2_2. \] 
\end{proof}
Now we have to take a look at the discrete version of the Gabor transform.  

\subsection{Gabor series}
In applications we have to replace the integrals in \eqref{eq:RecWQFT} by sums, which results in a discretization that leads to lattices in the space-frequency-parameter. In the following we define a lattice where each cell has the volume $1$. We take some elements $\alpha, \beta \in \mathbb{R}^+$ such that $\alpha \beta = 1$ and consider the lattice $\Lambda \subseteq \mathbb{R}^2 \times \mathbb{R}^2$ generated by $\alpha$ and $\beta$
\[ \Lambda = \{ \lambda = \alpha \qm + \beta \qn, \quad \qm, \qn \in \mathbb{Z}^2\}. \] 
Introducing a coordinate $\qb \in \mathbb{R}^2$ such that $\qb(\alpha) = 1$, then we have for the dual coordinate $\omega(\beta) = 1,  \omega \in \mathbb{R}^2$ in the frequency plane. This definition gives a lattice where the volume of a cell equals $1$, the so-called \emph{critical density.} The next proposition shows, that it is possible to reconstruct a quaternionic-valued signal with the help of shifted and modulated Gaussian functions in that lattice. 
\begin{prop}\label{prop:31}
For an arbitrary quaternion-valued sequence $\{c_\lambda \} \in \ell^2(\Lambda)$ with $\lambda = (\qb, \omega)$ the series
\begin{equation}\label{eq:Gaborseries} f(\qx) = \sum_{\lambda \in \Lambda} \exp(2 \pi i x_1 \omega_1) c_\lambda \exp(- \pi |\qx - \qb|^2) \exp(2 \pi j x_2 \omega_2) = \sum_{\lambda \in \Lambda} c_\lambda \qe'_\lambda(\qx) \end{equation}
converges in $L_2$ and the following inequality holds
\begin{align} \|f\|_2 \leq \sigma_0^2 \left( \sum_{\lambda_j = (b_2, \omega_2)} |c_{\lambda_j}|^2 \right)^{1/2} \left( \sum_{\lambda_i = (b_1, \omega_1)} |c_{\lambda_i}|^2 \right)^{1/2}, \end{align}
where $\sigma_0 = \sum_{n \in \mathbb{Z}} \exp \left(- \frac{\pi n^2}{2} \right)$.
\end{prop}

\begin{proof}
For $\lambda = (\qb, \omega)$ and $\mu = (\qd, \eta)$ we have
\begin{align}
&\|f\|_2^2 \nonumber \\
&= \Sc \int_{\mathbb{R}^2} \sum_{\lambda \in \Lambda} c_{\lambda} \qe_{\lambda}'(\qx) \sum_{\mu \in \Lambda} \overline{c_{\mu} \qe_{\mu}'(\qx)} \dx^2 \qx \nonumber \\
&= \Sc \int_{\mathbb{R}^2} \sum_{\lambda \in \Lambda} \exp(2 \pi i x_1 \omega_1) c_\lambda \exp(- \pi |\qx - \qb|^2) \exp(2 \pi j x_2 \omega_2) \nonumber \\
&\qquad \qquad \sum_{\mu \in \Lambda} \exp(- 2 \pi j x_2 \eta_2) \exp(- \pi |\qx - \qd|^2) \overline{c_\mu} \exp(-2 \pi i x_1 \eta_1) \dx^2 \qx \nonumber \\
&= \Sc \int_{\mathbb{R}^2} \sum_{\lambda, \mu \in \Lambda} \exp(2 \pi i x_1 \omega_1) c_\lambda \exp(- \pi |\qx - \qb|^2) \exp(2 \pi j x_2( \omega_2 - \eta_2)) \exp(- \pi |\qx - \qd|^2) \overline{c_\mu} \exp(-2 \pi i x_1 \eta_1) \dx^2 \qx \label{eq:innerj}.
\end{align}
Computing the $x_2$-integral of \eqref{eq:innerj} gives
\begin{align}
&\int_{\mathbb{R}} \exp(- \pi |x_2 - b_2|^2) \exp(2 \pi j x_2 \omega_2) \exp(-2 \pi j x_2 \eta_2) \exp(- \pi |x_2 - d_2|^2) \dx x_2 \nonumber \\
&=\int_{\mathbb{R}} \exp(- 2 \pi x_2^2 + 2 \pi b_2 x_2 - \pi b_2^2 + 2 \pi d_2 x_2 - \pi d_2^2 + 2 \pi j x_2 (\omega_2 - \eta_2)) \dx x_2 \nonumber \\
&=\exp(- \pi (b_2^2 + d_2^2)) \int_{\mathbb{R}} \exp(- 2 \pi (x_2^2 - x_2( b_2 + d_2 + \omega_2 j - \eta_2 j))) \dx x_2 \nonumber\\
&=\exp(- \pi (b_2^2 + d_2^2) + \frac{\pi (b_2 + d_2 + \omega_2 j - \eta_2 j)^2}{2}) \nonumber\\
&=\exp(- \pi j (b_2 + d_2)(\omega_2 - \eta_2) - \frac{\pi}{2} ((b_2 - d_2)^2 + (\omega_2 - \eta_2)^2)) \nonumber\\
&= \exp(- \pi j (b_2 + d_2)(\omega_2 - \eta_2) - |\lambda_j - \mu_j|^2) \nonumber\\
&= \exp(- \pi j (b_2 - d_2)(\omega_2 - \eta_2) - |\lambda_j - \mu_j|^2) \label{eq:period}.
\end{align}
The last line \eqref{eq:period} holds true, since we have chosen our lattice in a manner that we can use periodicity conditions. We insert this in \eqref{eq:innerj} and get with the help of the cyclic multiplication rule \eqref{eq:SC}:
\begin{align} 
&\Sc \int_{\mathbb{R}} \sum_{\lambda, \mu \in \Lambda} \exp(2 \pi i x_1 \omega_1) c_\lambda \exp(- \pi (x_1 - b_1)^2)  \exp(- \pi j (b_2 - d_2)(\omega_2 - \eta_2) - \frac{\pi}{2} ((b_2 - d_2)^2 + (\omega_2 - \eta_2)^2))  \label{innerj2} \nonumber \\[1em]
&\qquad  \exp(- \pi (x_1 - d_1)^2) \overline{c_\mu} \exp(-2 \pi i x_1 \eta_1) \dx x_1 \nonumber \\[1em]
&= \Sc \int_{\mathbb{R}} \sum_{\lambda, \mu \in \Lambda} c_\lambda \exp(- \pi j (b_2 - d_2)(\omega_2 - \eta_2) - \frac{\pi}{2} ((b_2 - d_2)^2 + (\omega_2 - \eta_2)^2)) \overline{c_\mu} \nonumber  \\[1em]
&\qquad \exp(- \pi (x_1 - b_1)^2) \exp(- \pi (2 \pi i x_1 (\omega_1 - \eta_1))) \exp(- \pi (x_1 - d_1)^2) \dx x_1 \nonumber \\[1em]
&= \Sc \sum_{\lambda, \mu \in \Lambda} c_\lambda \exp(- \pi j (b_2 - d_2)(\omega_2 - \eta_2) - \frac{\pi}{2} ((b_2 - d_2)^2 + (\omega_2 - \eta_2)^2)) \overline{c_\mu} \nonumber \\[1em]
& \qquad \exp(- \pi i (b_1 + d_1)(\omega_1 - \eta_1) - \frac{\pi}{2}((b_1 - d_1)^2 + (\omega_2 - \eta_2)^2)). \nonumber
\end{align}
For the estimate we incorporate the discrete Young inequality for a convolution with $r = p = 2$ and $q = 1$:
\begin{align}\label{eq:convo}
\sum_{m \in \mathbb{Z}^4} \left| \sum_{n \in \mathbb{Z}^4} a_n b_{m - n} \right|^2 \leq \left( \sum_{n \in \mathbb{Z}^4} |a_n| \right)^2 \sum_{m \in \mathbb{Z}^4} |b_m|^2 
\end{align}
and the functions
\begin{align}\label{eq:gg}
g_{\lambda_j} = \exp(\pi j b_2 \omega_2 - \frac{\pi}{2} (b_2^2 + \omega_2^2))  \qquad g_{\lambda_i} = \exp(\pi i b_1 \omega_1 - \frac{\pi}{2} (b_1^2 + \omega_1^2)) 
\end{align}
for $\lambda_j = (b_2, \omega_2) \in \mathbb{R} \times \mathbb{R}$ and $\lambda_i = (b_1, \omega_1) \in \mathbb{R} \times \mathbb{R}$. 
Moreover, we have
\begin{align}\label{eq:ggg}
\sum_{\lambda_j} |g_{\lambda_j}| = \sum_{\lambda_i} |g_{\lambda_i}| = \sum_{n \in \mathbb{Z}} \exp(- \frac{\pi}{2} n^2) = \sigma_0^2.
\end{align}
From this we obtain for the norm estimate of $f$
\begin{align}
\|f\|^4_2 \nonumber &\leq \big| \sum_{\lambda, \mu \in \Lambda} c_{\lambda_i}{c_{\lambda_j}}  \exp(- \pi j (b_2 - d_2)(\omega_2 - \eta_2) - \frac{\pi}{2} ((b_2 - d_2)^2 + (\omega_2 - \eta_2)^2)) \overline{c_{\mu_j}} \nonumber \\
& \qquad \overline{c_{\mu_i}}  \exp(- \pi i (b_1 - d_1)(\omega_1 - \eta_1) - \frac{\pi}{2}((b_1 - d_1)^2 + (\omega_2 - \eta_2)^2)) \big|^2 \nonumber \\
&\leq \sum_{\lambda, \mu \in \Lambda} \big|c_{\lambda_j} \exp(- \pi j (b_2 - d_2)(\omega_2 - \eta_2)- \frac{\pi}{2}((b_2 - d_2)^2 + (\omega_2 - \eta_2)^2)) \overline{c_{\mu_j}}\big|^2 \nonumber\\
&\qquad \big|c_{\lambda_i} \overline{c_{\mu_i}} \exp(- \pi i (b_1 - d_1)(\omega_1 - \eta_1) - \frac{\pi}{2}((b_1 - d_1)^2 + (\omega_1 - \eta_1)^2))\big|^2 \nonumber\\
&\leq \big| \sum_{\lambda_j} \sum_{\mu_j} c_{\lambda_j} g_{\lambda_j - \mu_j} \overline{c_{\mu_j}} \big|^2 \big| \sum_{\lambda_i} \sum_{\mu_i} c_{\lambda_i} \overline{c_{\mu_i}} g_{\lambda_i - \mu_i} \big|^2 \nonumber\\
& \leq \sum_{\lambda_j} |c_{\lambda_j}|^2 \sum_{\lambda_j} \left| \sum_{\mu_j} g_{\lambda_j - \mu_j} \overline{c_{\mu_j}} \right|^2 \sum_{\lambda_i} |c_{\lambda_i}|^2 \sum_{\lambda_i} \left| \sum_{\mu_i}  \overline{c_{\mu_i}} g_{\lambda_i - \mu_i} \right|^2 \nonumber \\
&= \sum_{\lambda_j} |c_{\lambda_j}|^2 \left(\sum_{\lambda_j} |g_{\lambda_j}| \right)^2 \sum_{\mu_j} |c_{\mu_j}|^2 \sum_{\lambda_i} |c_{\lambda_i}|^2 \left( \sum_{\lambda_i} |g_{\lambda_i}| \right)^2 \sum_{\mu_i} |c_{\mu_i}|^2 \label{eq:z} \\
& = \sigma_0^8 \left( \sum_{\lambda_j} |c_{\lambda_j}|^2 \right)^2 \left( \sum_{\lambda_i} |c_{\lambda_i}|^2 \right)^2 \label{eq:z2}
\end{align}
In \eqref{eq:z} we use the convolution estimate \eqref{eq:convo}. For the estimate of \eqref{eq:z2} we use equation \eqref{eq:ggg}.\\
All this leads into
\[ \|f\|_2 \leq \sigma_0^2 \left( \sum_{\lambda_j} |c_{\lambda_j}|^2 \right)^{1/2} \left( \sum_{\lambda_i} |c_{\lambda_i}|^2 \right)^{1/2}. \]
\end{proof}


\section{Quaternionic Zak transform}
The quaternionic Zak transform is the natural extension of the Zak transform into the two-sided quaternionic setting. The classic Zak transform is useful in the analysis of the Gabor transform, of course, the extension keeps similar properties \cite{Zakquat, Grchenig2001}. In harmonic analysis the Zak transform is also known as Weil-Brezin map.  For our results we will invoke the quaternionic Zak transform several times.
\begin{defn}[Quaternionic Zak transform]
The \emph{quaternionic Zak transform} of a quaternion-valued function $f \in L^2(\mathbb{R}^2,\mathbb{H})$ is defined as
\[ \ZT_q f(\qx, \omega) = \sum_{\qm \in \mathbb{Z}^2} \exp(2 \pi i m_1 \omega_1) f(\qx - \qm) \exp(2 \pi j m_2 \omega_2) \qquad \qx, \omega \in \mathbb{R}^2. \]
\end{defn}

In what follows we show some interesting properties of the quaternionic Zak transform. 
\begin{prop}[Quasiperiodicity conditions]
The transform $\ZT_q$ fulfills
\[ \ZT_q f(\qx,\omega + 1) = \ZT_q f(\qx,\omega), \qquad \ZT_q f(\qx + 1, \omega) = \exp(2 \pi i  \omega_1) \ZT_q f(\qx, \omega) \exp(2 \pi j \omega_2). \]
\end{prop}

The next remark concerns the zeros of the quaternionic Zak transform of a modulated and translated Gaussian window. 
\begin{rem}\label{rem:1}
We have for $\mu = (\qp, \theta)$:
\begin{align*}
&\ZT_q \qe'_{\mu}(\qx, \omega)\\[1ex]
& = \sum_{\qm \in \mathbb{Z}^2}\exp(2 \pi i m_1 \omega_1) \exp(2 \pi i \theta_1(x_1 - m_1))\exp(- \pi (x_1 - m_1 - p_1)^2)\\[1ex]
&\qquad \exp(- \pi (x_2 - m_2 - p_2)^2) \exp(2 \pi j \theta_2(x_2 - m_2))(\exp(2 \pi j m_2 \omega_2))\\[1ex]
& = \exp(2 \pi i \theta_1 x_1) \Big(\sum_{\qm \in \mathbb{Z}^2} \exp(2 \pi i (m_1 \omega_1 - \theta_1 m_1)) \exp(- \pi (x_2 - m_2 - p_2)^2)\\[1ex]
&\qquad \exp(- \pi (x_1 - m_1 - p_1)^2)\exp(2 \pi j (m_2 \omega_2 - \theta_2 m_2))\Big) \exp(2 \pi j \theta_2 x_2)\\[1ex]
& = \exp(2 \pi i \theta_1 x_1) \Big(\sum_{\qm' \in \mathbb{Z}^2} \exp(2 \pi i (m'_1 \omega_1 + p_1 \omega_1 - \theta_1 m'_1 - \theta_1 p_1)) \exp(- \pi (x_1 + m'_1)^2)\\[1ex]
&\qquad \exp(- \pi (x_2 + m'_2)^2) \exp(2 \pi j (m'_2 \omega_2 + p_2 \omega_2 - \theta_2 m'_2 - \theta_2 p_2))\Big)\exp(2 \pi j \theta_2 x_2)\\[1ex]
& = \exp(2 \pi i (\theta_1 x_1 + p_1 \omega_1))\exp(- \pi x_1^2) \Big(\sum_{\qm' \in \mathbb{Z}^2} \exp(2 \pi i m'_1 \omega_1 - 2 \pi x_1 m'_1 - \pi m'^1_2)\\[1ex]
&\qquad \exp(2 \pi j m'_2 \omega_2 - 2 \pi x_2 m'_2 - \pi m'^2_2)\Big)\exp(- \pi x_2^2)\exp(2 \pi j (\theta_2 x_2 + p_2 \omega_2))\\[1ex]
&= \exp(2 \pi i (\theta_1 x_1 + p_1 \omega_1))\exp(- \pi x_1^2) \Theta_i(\omega_1 + i x_1) \Theta_j (\omega_2 + j x_2)\exp(- \pi x_2^2)\exp(2 \pi j (\theta_2 x_2 + p_2 \omega_2))
\end{align*}
with the two Theta series
\begin{align*}
\Theta_i(z_1) = \sum_{m \in \mathbb{Z}} \exp(2 \pi i m z - \pi m^2) \qquad \text{and} \qquad \Theta_j(z_2) = \sum_{m \in \mathbb{Z}} \exp(2 \pi j m z - \pi m^2)   \qquad \text{with $z_1 \in \mathbb{C}_i, z_2 \in \mathbb{C}_j$}.
\end{align*}
\end{rem}

For the involved Theta series we have the following properties.
\begin{rem}\label{rem:rem2}
The Theta series are holomorphic, have a simple zero in the cube $Q$ only for $z_1 = 1/2 + i/2$ and $z_2 = 1/2 + j/2$, respectively \cite{Bellman1961}. Moreover, the Theta series satisfy
\[  \Theta_i(z_1+1) = \Theta_i(z_1), \qquad \qquad \Theta_i(z_1+i) =  \exp(\pi - 2 \pi i z_1) \Theta_i(z_1) \]
and
 \[  \Theta_j(z_2+1) = \Theta_j(z_2), \qquad \qquad \Theta_j(z_2+j) = \exp(\pi - 2 \pi j z_2) \Theta_j(z_2). \]
\end{rem}

We also have the following continuity property of the quaternionic Zak transform. 
\begin{lem}\label{lem:xx}
For any $f \in W_0$ the quaternionic Zak transform is continuous.
\end{lem}
\begin{proof}
Given $\varepsilon > 0$, there exists a $N> 0$ such that\\
\[ \sum_{|\qk| > N} \|f \cdot T_{\qk} \mathcal{X}_{Q}\|_\infty < \frac{\varepsilon}{4}. \]
Then, the main term $\sum_{|\qk| \leq N} \exp(2 \pi i k_1 \omega_1) f(\qx -  \qk) \exp(2 \pi j k_2 \omega_2)$ is uniformly continuous on compact sets of $\mathbb{R}^4,$ and there exists a $\delta > 0$ such that
\[ \left| \sum_{|\qk| \leq N} \exp(2 \pi i k_1 \omega_1) f(\qx - \qk) \exp(2 \pi j k_2 \omega_2) - \sum_{|\qk| \leq N} \exp(2 \pi i k_1 \xi_1) f(\qy - \qk) \exp(2 \pi j k_2 \xi_2)\right| < \frac{\varepsilon}{2} \]
whenever $|\qx - \qy| + |\omega - \xi| < \delta.$ As a consequence, $|\ZT_q f(\qx, \omega) - \ZT_q f(\qy, \xi)| < \varepsilon$ and $\ZT_q f$ is continuous.
\end{proof}

Furthermore, the quaternionic Zak transform is a unitary mapping.
\begin{lem}
The quaternionic Zak transform is a unitary operator mapping $L^2(\mathbb{R}^2, \mathbb{H})$ onto $L^2(Q \times Q, \mathbb{H})$.
\end{lem}
\begin{proof}
We consider the functions $F_\qm(\qx,\omega) = \exp(2 \pi i m_1 \omega_1) f(\qx - \qm)(\exp(2 \pi j m_2 \omega_2))$ with $\qm \in \mathbb{Z}^2$. These functions belong to $L^2(Q \times Q,\mathbb{H})$. 
\begin{align*}
\sum_{\qm \in \mathbb{Z}^2} \|F_\qm\|^2_{L^2(Q \times Q)} &= \sum_{\qm \in \mathbb{Z}^2} \langle F_\qm, F_\qm \rangle\\
&= \sum_{\qm \in \mathbb{Z}^2} \Sc \int_{Q} \int_{Q} \exp(2 \pi i m_1 \omega_1) f(\qx - \qm) \exp(2 \pi j m_2 \omega_2)\\
&\qquad  \exp(-2 \pi j m_2 \omega_2)\overline{f(\qx - \qm)} \exp(- 2 \pi i m_1 \omega_1) \dx^2 \qx \dx^2 \omega\\
&= \sum_{\qm \in \mathbb{Z}^2} \Sc \int_{Q} \int_{Q} \exp(2 \pi i m_1 \omega_1)f(\qx - \qm) \exp(2 \pi j m_2 \omega_2)\\
&\qquad \exp(-2 \pi j m_2 \omega_2) \overline{f(\qx - \qm)} \exp(-2 \pi i m_1 \omega_1) \dx^2 \qx \dx^2 \omega\\
&= \sum_{\qm \in \mathbb{Z}^2} \int_{Q} \int_{Q}|f(\qx - \qm)|^2 \dx^2 \qx \dx^2 \omega\\
&= \alpha^{-2} \|f\|^2_2
\end{align*}
For $\qm \neq \qn$, we have with the cyclic multiplication \eqref{eq:SC}
\begin{align*}
\sum_{\qm,\qn \in \mathbb{Z}^2} \langle F_\qm, F_\qn \rangle
&= \sum_{\qm, \qn \in \mathbb{Z}^2} \Sc \int_{Q} \int_{Q} \exp(2 \pi i m_1 \omega_1) f(\qx - \qm) \exp(2 \pi j m_2 \omega_2)\\
&\qquad  \exp(-2\pi j n_2 \omega_2)\overline{f(\qx - \qn)}\exp(- 2 \pi i n_1 \omega_1) \dx^2 \qx \dx^2 \omega\\
&= \Sc \int_{Q} \int_{Q} \exp(2 \pi i m_1 \omega_1) f(\qx - \qm) \exp(2 \pi j (m_2 - n_2)\omega_2) f(\qx - \qn) \\
& \qquad \exp(- 2 \pi i n_1 \omega_1) \dx^2 \qx \dx^2 \omega\\
&= \Sc \int_{Q} \int_{Q} f(\qx -\qm) \exp(2 \pi j (m_2 - n_2)\omega_2)f(\qx - \qn) \exp(-2 \pi i n_1 \omega_1)\\
&\qquad \exp(2 \pi i m_1 \omega_1) \dx^2 \qx \dx^2 \omega\\
&= 0
\end{align*}
By combining the above results we get that $\sum_{\qm \in \mathbb{Z}^2} F_{\qm}$ converges in $L^2(Q \times Q, \mathbb{H})$ and
\[ \|\sum_{\qm \in \mathbb{Z}^2} F_{\qm} \|^2_{L^2(Q \times Q)} = \sum_{\qm \in \mathbb{Z}^2} \|F_\qm\|^2_{L^2(Q \times Q)} = \|f\|^2_2 \]
Thus $\ZT_q$ is an isometry from $L^2(\mathbb{R}^2,\mathbb{H})$ to $L^2(Q \times Q,\mathbb{H})$. For the rest of this proof we use the Gabor orthonormal basis
\[ \mathcal{X}_{\qm,\qn} = \exp(2 \pi i m_1 x_1) \mathcal{X}(\qx - \qn) \exp(2 \pi j m_2 x_2) \]
where $\mathcal{X}$ denotes the characteristic function of $Q$, defined by $\mathcal{X}(\qx) = 1$ for $\qx \in Q$ and $\mathcal{X}(\qx) = 0$ otherwise. Direct calculation tells us that
\[ \ZT_q \mathcal{X}_{\qm,\qn}(\qx, \omega) = \exp(2 \pi i (m_1 x_1 - n_1 \omega_1))\ZT_q \mathcal{X} \exp(2 \pi j (m_2 x_2 - n_2 \omega_2)) \]
and $\ZT_q \mathcal{X} = 1$ with $(\qx, \omega) \in Q \times Q$. Since $\ZT_q \mathcal{X}_{\qm,\qn}$ forms an orthonormal basis with the scalar inner product \eqref{eq:Sc}, $\ZT_q$ maps an orthonormal basis of $L^2(\mathbb{R}^2,\mathbb{H})$ to an orthonormal basis of $L^2(Q\times Q, \mathbb{H})$. Thus $\ZT_q$ is unitary. 
\end{proof}

The next lemma shows that it is possible to reconstruct the signal $f$ only by the values of $\ZT_q f$ on the cube $Q$. This means we only have to analyze the Zak transform of the signal in the cube to reach global information about $f$.
\begin{lem}
If $f \in L^2(\mathbb{R}^2, \mathbb{H})$, then 
\[ f(\qx) = \int_Q \ZT_q f(\qx, \omega) \dx^2 \omega \]
\end{lem}
\begin{proof}
We use the definition of the Zak transform and obtain
\begin{align*}
&\int_{Q} \ZT_q f(\qx, \omega) \dx^2 \omega\\
&=\int_{Q} \sum_{\qm \in \mathbb{Z}^2} \exp(2 \pi i \omega_1 m_1) f(\qx - \qm) \exp(2 \pi j \omega_2 m_2 ) \dx^2 \omega\\
&=\int_{Q} f(\qx) \dx^2 \omega + \int_{Q} \sum_{\qm \neq 0} \exp(2 \pi i \omega_1 m_1) f(\qx - \qm)\exp(2 \pi j \omega_2 m_2)) \dx^2 \omega
\end{align*}
The first integral is just $f(\qx)$. In order to calculate the second integral we interchange the order of integration and summation and get
\begin{align*}
&\int_{Q} \sum_{\qm \neq 0} \exp(2 \pi i \omega_1 m_1) f(\qx - \qm) \exp(2 \pi j \omega_2 m_2)) \dx^2 \omega\\
&=\sum_{\qm \neq 0} \int_{[0,1]} \exp(2 \pi i \omega_1 m_1) \dx \omega_1 \  f(\qx - \qm)\  \int_{[0,1]} \exp(2 \pi j \omega_2 m_2) \dx \omega_2\\
&= 0.
\end{align*}
This completes our proof.
\end{proof}

\section{Relaxed Expansion}
We show now that an arbitrary function $f \in W_0$ can be expanded in a unique Gabor series as in \eqref{eq:Gaborseries} when one element in the middle of a cell is added. Denote by $\# = (1/2, 1/2)$ a point in the middle of the cell. We call the set $\Lambda^{\#} = \Lambda \cup \{\#\}$ the \emph{relaxed lattice}. We use similar concepts as in \cite{Palamodov}.\\

First of all, we need to define the sharp Poisson functional, which will be a necessary tool for the results.
\begin{defn}[Sharp Poisson functional]
The sharp Poisson functional is the series
\begin{align}
\gamma^{\#}(f) = (i \Theta_i(0))^{-1} \ZT_q f(\#) (j \Theta_j(0))^{-1}  \stackrel{\cdot}{=} \frac{\ZT_q f(\#)}{\ZT_q \qe'_0}
\end{align}
\end{defn}
The sharp Poisson functional has the following important property, which is needed to find the coefficient at the sharp point $\#$.
\begin{lem}
We have $\gamma^{\#}(\qe'_{\#}) = 1$ and $\gamma^{\#}(\qe'_{\lambda}) = 0$ for arbitrary $\lambda \in \Lambda$.
\end{lem}
\begin{proof}
The use of the Euler formula for quaternions gives for $\gamma^{\#}(\qe'_{\#})$
\begin{align*}
&(i \Theta_i(0))^{-1} \ZT_q \qe'_{1/2,1/2}(1/2, 1/2) (\Theta_j(0) j)^{-1} \\
&= (i \Theta_i(0))^{-1} \sum_{\qm \in \mathbb{Z}^2} \exp( \pi i m_1) \exp(- \pi m_1^2) \exp( \frac{i \pi}{2} + \pi i m_1) \exp(\pi j m_2) \exp(- \pi m_2^2) \exp(\frac{j \pi}{2} + \pi j m_2)(\Theta_j(0) j)^{-1}\\
&= (i \Theta_i(0))^{-1}\sum_{\qm \in \mathbb{Z}^2} \exp(- \pi m_1^2) (\cos(\pi/2) + i \sin(\pi/2)) (\cos(\pi/2) + j \sin(\pi/2)) \exp(- \pi m_2^2) (\Theta_j(0)j)^{-1}\\
&= (i \Theta_i(0))^{-1}(i \Theta_i(0)) (\Theta_j(0) j)(\Theta_j(0)j)^{-1}\\
&= 1
\end{align*}
And for $\lambda = (\qb, \omega)$   
\begin{align*}
&\gamma^{\#}(\qe'_{\lambda})\\
&=(i \Theta_i(0))^{-1} \ZT_q \qe'_{\lambda}(1/2, 1/2) (\Theta_j(0)j)^{-1}\\
&=(i \Theta_i(0))^{-1} \exp(\pi i (b_1 + \omega_1)) \exp(- \pi/4) \Theta_i(1/2 + i/2) \Theta_j(1/2 + j/2) \exp(\pi j (b_2 + \omega_2)) (j \Theta_j(0))^{-1}\\
&= 0
\end{align*}
since $\Theta_i$ and $\Theta_j$ are zero at the sharp point, see Remark \ref{rem:rem2}.
\end{proof}

In the following theorem we show that the Sharp Poisson functional is well-defined.
\begin{thm}
For any $f \in W_0$ there exists a family of continuous functionals $\gamma^{\lambda}$, $\lambda \in \Lambda$ in the space
$W_0$, such that
\begin{align}
f(\qx) = \sum_{\lambda \in \Lambda^{\#}} \exp(2 \pi i x_1 \omega_1) \gamma^{\#}(f) \exp(- \pi |\qx - \qb|^2) \exp(2 \pi j x_2 \omega_2) = \sum_{\lambda \in \Lambda^{\#}} \gamma^{\#}(f) \qe'_{\lambda}(\qx)
\end{align}
\end{thm}

\begin{proof}
By Lemma \ref{lem:xx}, $\gamma^{\#}$ is well-defined in $W_0(\mathbb{R}^2)$. Set
\begin{align}
f_{\#}(\qx) = f(\qx) - \exp(- \pi (x_1 - \frac{1}{2})^2) \exp(\pi i x_1 \omega_1) \gamma^{\#}(f) \exp(- \pi (x_2 - \frac{1}{2})^2) \exp(\pi j x_2 \omega_2) 
\end{align}
and define $F(\qx, \omega) = \ZT_q f_{\#} / \ZT_q \qe'_{0} = \exp(\pi x_1^2) (\Theta_i(\omega_1 + i x_1))^{-1} \ZT_q f_{\#}(\qx, \omega) (\Theta_j( \omega_2 + j x_2)) \exp(\pi x_2^2)$. By Lemma \ref{lem:x3}, $F$ is square integrable and has the following double Fourier series expansion
\begin{align}\label{eq:Fqxomega}
F(\qx, \omega) = \sum_{\mu = (\qp, \eta) \in \Lambda} \exp(2 \pi i (x_1 p_1 + \omega_1 \eta_1)) c_{\mu} \exp(2 \pi j (x_2 p_2 + \omega_2 \eta_2)) 
\end{align} 
By Proposition \ref{prop:31}, the series $\sum \exp(2 \pi i x_1 \omega_1) c_\mu \exp(- \pi |\qx - \qb|^2) \exp(2 \pi j x_2 \omega_2) = \sum c_\mu \qe'_\mu$ converges to a function $g$.\\
On the other hand we have with Remark \ref{rem:1}
\begin{align*}
&\ZT_q f_{\#}(\qx, \omega)\\
&= \exp(- \pi x_1^2) \Theta_i(\omega_1 + i x_1) \left[\sum_{\mu = (\eta, \qp) \in \Lambda} \exp(2 \pi i (x_1 p_1 + \omega_1 \eta_1)) c_{\mu} \exp(2 \pi j (x_2 p_2 + \omega_2 \eta_2)) \right] \Theta_j(\omega_2 + j x_2) \exp(- \pi x_2^2)\\
&= \sum_{\mu = (\eta, \qp) \in \Lambda} c_{\mu} \ZT_q \qe'_{\mu}.
\end{align*}
Which means by the unitarity of $\ZT_q$ that $f_{\#} = g$, set $\gamma^\lambda(f) = c_{\mu}$ for $\mu \in \Lambda$.
\end{proof}

\begin{lem}\label{lem:x3}
We have $F(\qx, \omega) \in L^2(Q \times Q)$.
\end{lem}
\begin{proof}
The function $f_{\#}$ is an element of $W_0$. By Lemma \ref{lem:xx}, $\ZT_q f_{\#}(\qx, \omega)$ is continuous. It vanishes at the sharp point since $\gamma^{\#}(f_{\#}) = 0$. Therefore, with $z_1 = x_1 + i \omega_1$ and $z_2 = x_2 + j \omega_2$, $|\ZT_q f(z_1, z_2)| \leq |z_1 - i/2 - 1/2|\,|z_2 - j/2 - 1/2|$. This implies
\begin{align}
|F(z_1, z_2)| = \left| (\Theta_i(z_1))^{-1} \exp(\pi \qx^2) \ZT_q f_{\#} (\Theta_j(z_2))^{-1} \right| \leq \frac{|z_1 - i/2 - 1/2|\,|z_2 - j/2 - 1/2|}{|z_1 - i/2 - 1/2|\,|z_2 - j/2 - 1/2|} =  1.
\end{align}
And therefore it is square-integrable in the cube. 
\end{proof}

In the next theorem we are going to show that this expansion is linear independent. 
\begin{thm}
If $\{c_\lambda \} \in \ell^2(\Lambda)$, $\qu \in \mathbb{H}$ and
\begin{equation}\label{eq:ad}
\exp(\pi i x_1 \omega_1) \qu  \exp(- \pi |\qx - \frac{1}{2}|^2) \exp(\pi j x_2 \omega_2) + \sum_{\lambda \in \Lambda} \exp(2 \pi i x_1 \omega_1)  c_{\mu} \exp(- \pi |\qx - \qb|^2) \exp(2 \pi j x_2 \omega_2) = 0
\end{equation}
or in short notation
\[ \qu \qe'_{\#}(\qx) + \sum_{\lambda \in \Lambda} c_\mu \qe'_{\mu}(\qx) = 0  \]
then $c_{\lambda} = 0$ for all $\lambda \in \Lambda$.
\end{thm}

\begin{proof}
By Proposition \ref{prop:31}, the series in \eqref{eq:ad} converges in $L^2(\mathbb{R}^2, \mathbb{H})$. We apply the Zak transform and get
\begin{align*}
&- \qu \ZT_q \qe'_{\#}(\qx, \omega)\\
& = \exp(- \pi x_1^2) \Theta_i(\omega_1 + i x_1) \left[\sum_{\lambda \in \Lambda} \exp(2 \pi i x_1 \omega_1) c_{\lambda} \exp(- \pi |\qx - \qb|^2) \exp(2 \pi j x_2 \omega_2)\right] \Theta_j(\omega_2 + j x_1) \exp(- \pi x_2^2)
\end{align*}
The sum converges to a function $g(\qx)$. Therefore 
\begin{align}
g(\qx) = -\frac{ \qu \ZT_q \qe'_{\#}}{\ZT_q \qe'_0} =  - \exp(\pi x_1^2) (\Theta_i(\omega_1 + i x_1))^{-1} \left[ \qu \ZT_q \qe'_{\#}(\qx, \omega) \right] (\Theta_j(\omega_2 + j x_2))^{-1} \exp(\pi x_2^2)
\end{align}
We already know that $\ZT_q \qe'_0(\qx, \omega) = 0$ for the sharp point $\#$, whereas $\ZT_q \qe'_{\#}(\qx, \omega) \neq 0$. Therefore $g(\qx)$ can only be an element of $L^2(\mathbb{R}^2, \mathbb{H})$ if and only if $\qu$ is zero. This means that $g(\qx) = 0$ and therefore $c_{\lambda} = 0$ for $\lambda \in \Lambda$. 
\end{proof}

\begin{rem}
The last Theorem shows that the series expansion on the relaxed lattice is unique. 
\end{rem}

\begin{cor}
We get the coefficients via 
\[ \gamma^{\lambda}(f) = \int_{Q} \int_{Q} \frac{\ZT_q f_{\#}(\qx, \omega)}{\ZT_q \qe'_{\lambda}(\qx, \omega)} \dx^2 \qx \dx^2 \omega \]
\end{cor}

\begin{proof}
We have 
\begin{align*}
\frac{\ZT_q f_{\#}(\qx, \omega)}{\ZT_q \qe'_{0}(\qx, \omega)} = F(\qx, \omega) = \sum_{\lambda = (\qp, \eta) \in \Lambda} \exp(2 \pi i (x_1 p_1 + \omega_1 \eta_1)) c_{\lambda} \exp(2 \pi j (x_2 p_2 + \omega_2 \eta_2)) 
\end{align*} 
Extracting the coefficients $c_{\lambda}$ by using the inverse Fourier transform gives
\begin{equation*}
c_{\lambda} = \int_Q \int_Q \exp(-2 \pi i (x_1 l_1 + \omega_1 k_1)) \frac{\ZT_q f_{\#}(\qx, \omega)}{\ZT_q \qe'_{0}(\qx, \omega)} \exp(-2 \pi j (x_2 l_2 + \omega_2 k_2)) \dx^2 \qx \dx^2 \omega.
\end{equation*}
Using Remark \ref{rem:1} we obtain the desired formula.
\end{proof}

\section{Further research}
The requirement of our signal to be an element of $W_0(\mathbb{R}^2)$ is a strong restriction. The main condition for our proof is $\ZT_q f_{\#}$ to be continuous and square integrable. A possible extension for further research could be Sobolev spaces or modulation spaces. Another interesting topic for further research could be the investigation of the convergence properties of the coefficients if more points are added to the lattice.  

\section{Acknowledgements}
The author gratefully acknowledges the many helpful suggestions of S. Bernstein, P. Cerejeiras and U. Kähler during the preparation of the paper. This work was supported through the ERASMUS program and Portuguese funds through the CIDMA - Center for Research and Development in Mathematics and Applications, and the Portuguese Foundation for Science and Technology (``FCT - Funda\c{c}\~ao para a Ci\^encia e a Tecnologia''), within project UID/MAT/0416/2013. 

\bibliographystyle{plain}
\bibliography{References}

\end{document}